\documentclass{amsproc}
\usepackage{epsfig,amscd,amssymb,amsmath,amsfonts}
\usepackage{xypic}

\newtheorem{theorem}{Theorem}[section]
\newtheorem{lemma}[theorem]{Lemma}
\newtheorem{proposition}{Proposition}[section]
\theoremstyle{definition}
\newtheorem{definition}[theorem]{Definition}

\newtheorem{example}[theorem]{Example}

\theoremstyle{remark}
\newtheorem{remark}[theorem]{Remark}

\numberwithin{equation}{section}

\begin{document}

\title{Bar Simplicial Modules and Secondary Cyclic  (Co)homology}
\author{Jacob Laubacher }
\address{
Department of Mathematics and Statistics, Bowling Green State University, Bowling Green, OH 43403 }
\email{jlaubac@bgsu.edu}

\author{Mihai D. Staic}
\address{
Department of Mathematics and Statistics, Bowling Green State University, Bowling Green, OH 43403 }
\address{Institute of Mathematics of the Romanian Academy, PO.BOX 1-764, RO-70700 Bu\-cha\-rest, Romania.}
\thanks{The second author was partially supported by a grant of the Romanian National Authority for Scientific Research, CNCS-UEFISCDI, project number PN-II-ID-PCE-2011-3-0635, contract nr. 253/5.10.2011.}
\email{mstaic@bgsu.edu}


\author{Alin Stancu}
\address{Department of Mathematics, Columbus State  University, Columbus, GA 31907}
\email{stancu\_alin1@columbusstate.edu}

\subjclass[2010]{Primary  16E40, Secondary  18G30, 19D55}
\date{January 1, 1994 and, in revised form, June 22, 1994.}


\keywords{Hochschild  cohomology, Cyclic cohomology, simplicial $k$-modules}

\begin{abstract} In this paper we study the simplicial structure of the complex $C^{\bullet}((A,B,\varepsilon); M)$, associated to the secondary Hochschild cohomology. The main ingredient is the simplicial object $\mathcal{B}(A,B,\varepsilon)$, which plays a role equivalent to that of the bar resolution associated to an algebra. We also introduce the secondary cyclic (co)homology and establish some of its properties (Theorems \ref{Connes} and \ref{Connes2}).
\end{abstract}

\maketitle


%

\section*{Introduction}

Hochschild cohomology of an associative $k$-algebra $A$ with coefficients in the $A$-bimodule $M$, $H^\bullet(A, M)$, was introduced in \cite{h} as the homology of an explicit complex. Hochschild's motivation was the study of certain extensions of algebras and the characterization of separability of algebras. Cartan and Eilenberg's realization that Hochschild cohomology is an  example of a derived functor meant that resolutions can be used to study it (see \cite{CE}). The main ingredient used for the description of the Hochschild (co)homology became the so called bar resolution. A few years later, Gerstenhaber discovered the importance of Hochschild cohomology in the study of deformations of algebras as well as its rich algebraic structure (see \cite{g1} and \cite{g2}). A first connection with geometry was noticed in \cite{hkr} where it was proved that, for a smooth algebra $A$, the Hochschild homology can be explicitly computed using the module of differential forms on $A$.  Later, Connes noticed that the complex used to define the Hochschild cohomology $HH^{\bullet}(A)=H^{\bullet}(A, A^*)$ admits an action of the cyclic group $C_{\bullet+1}$, and that this action is compatible with the differential (see \cite{c}, \cite{lo})). He used this fact to define the cyclic cohomology which, in the commutative case, is essentially equivalent to  the de Rham homology.   This was a pivotal point in the development of noncommutative geometry since cyclic (co)homology is defined for algebras that are not necessarily commutative.  

The secondary Hochschild cohomology groups, $H^n((A,B,\varepsilon);M)$, are associated to a $B$-algebra $A$ (determined by $\varepsilon:B\to A$) and an $A$-bimodule $M$, which is $B$-symmetric.  Their construction was motivated by an algebraic version of the second Postnikov invariant (see \cite{staic}). The secondary Hochschild  cohomology was used in \cite{sta} to study deformations of the $B$-algebra structures on
$A[[t]$.  Some of its properties were studied in \cite{s3}, where it was proved that the complex $C^{\bullet}((A,B,\varepsilon); A)$ is a multiplicative non-symmetric operad and that the secondary Hochschild cohomology $H^{\bullet}((A,B,\varepsilon);A)$ is a Gerstenhaber algebra. It was also proved there, for characteristic 0, that $H^{\bullet}((A,B,\varepsilon);A)$ admits a Hodge type decomposition. When $B=k$ one can easily see that the secondary Hochschild cohomology reduces to the usual Hochschild cohomology.

In this paper we study the simplicial structure of the complex $C^{\bullet}((A,B,\varepsilon);M)$, which defines the secondary cohomology.  For this we introduce the notion of simplicial module over a simplicial algebra.  The main example is an object $\mathcal{B}(A,B,\varepsilon),$ that we call the {\it Secondary Bar Simplicial Module.} In many ways $\mathcal{B}(A,B,\varepsilon)$ behaves like the bar resolution associated to a $k$-algebra $A$.  Using $\mathcal{B}(A,B,\varepsilon)$  we get a presentation of $H^{\bullet}((A,B,\varepsilon);M)$ as the homology of the complex associated to a simplicial $k$-module $Hom_{\mathcal{A}(A,B,\varepsilon)}(\mathcal{B}(A,B,\varepsilon), \mathcal{C}(M))$.
This approach also allows us to give natural constructions for the secondary  Hochschild  and secondary cyclic cohomologies  associated to the triple $(A,B,\varepsilon)$ ($HH^{\bullet}(A,B,\varepsilon)$ respectively  $HC^{\bullet}((A,B,\varepsilon);M)$). These two new cohomology theories have many of the nice properties one would expect. We mention here the functionality and the existence of a long exact sequence (of Connes' type) relating them.  We should point out that $HH^{n}(A,B,\varepsilon)$ is not a particular example of the group $H^{n}((A,B,\varepsilon);M)$, this being one place where it is important to work with simplicial modules and not just modules.

Another application of the  secondary bar simplicial module $\mathcal{B}(A,B,\varepsilon)$ is that it allows us to introduce the secondary homology groups $H_{\bullet}((A,B,\varepsilon);M)$ (as well as the Hochschild and cyclic homology groups associated to the triple $(A, B, \varepsilon)$). When $B=k$, we recover the usual Hochschild and cyclic homology. We expect that the groups introduced here, $HH_{\bullet}(A,B,\varepsilon)$ and $HC_{\bullet}(A,B,\varepsilon)$, have applications to geometry and we plan  to study this problem in a future paper.

\section{Preliminaries}

\subsection{Cyclic (co)homology}
Throughout this paper $k$ is a field, $\otimes=\otimes_k$, and all $k$-algebras have a multiplicative unit. 
Cyclic cohomology was introduced by Connes. For a comprehensive account of this topic see \cite{c}, \cite{k} or \cite{lo}.
We recall here that the cyclic cohomology is the cohomology of a certain sub-complex of the Hochschild complex. More precisely, let $A$ be an associative $k$-algebra and $A^{*}=Hom_k(A, k)$ with the  $A$-bimodule structure defined by $(afa')(x)=f(a'xa)$. The Hochschild coboundary on $C^\bullet(A, A^*)$ transfers to $C^\bullet(A):=C^{\bullet+1}(A, k)$, where it is given by the formula
\begin{center}$\displaystyle(bf)(a_0, a_1,\dots,a_{n+1})=\sum_{i=0}^n(-1)^if(\dots,a_ia_{i+1},\dots)+(-1)^{n+1}f(a_{n+1}a_0, \dots,a_n).$\end{center}
The operators $b':C^{n}(A)\rightarrow C^{n+1}(A)$ and $\lambda:C^{n}(A)\rightarrow C^{n}(A)$ defined by
\begin{center}$\displaystyle(bf)(a_0, a_1,\dots,a_{n+1})=\sum_{i=0}^n(-1)^if(\dots,a_ia_{i+1},\dots)$,\\$(\lambda f)(a_0,\cdots, a_n)=(-1)^nf(a_n,\cdots,a_0)$\end{center} satisfy the relation $(1-\lambda)b=b'(1-\lambda)$. If we set $C_\lambda^\bullet(A):=Ker (1-\lambda)$, then $(C_\lambda^\bullet, b)$ is a complex, called the cyclic complex of $A$.  The homology of this complex is called the cyclic cohomology of $A$ and is denoted by $HC^\bullet(A)$.

The canonical monomorphism $C_\lambda^\bullet(A)\rightarrow C^{\bullet}(A, A^*)$ from the cyclic cochain complex $C_\lambda^\bullet(A)$ to the Hochschild complex with coefficients in $A^*$ induces a morphism $I:HC^{\bullet}(A)\to HH^{\bullet}(A)$. Connes proved that, in characteristic 0, there exists a long exact sequence
\begin{eqnarray*}
...\rightarrow HC^{n-2}(A)\rightarrow HC^n(A) \stackrel{I}{\rightarrow} HH^n(A)\rightarrow HC^{n-1}(A)\rightarrow HC^{n+1}(A) \stackrel{I}{\rightarrow}...
\end{eqnarray*}

The ``dual" notion, called cyclic homology, appeared in the work of Tsygan (see \cite{TS} ) and Loday-Quillen (see \cite{LQ}) and it can be described as the homology of a quotient of the Hochschild complex $(C_\bullet(A, A), b)$. More precisely, let $C_n(A, A)=A^{\otimes(n+1)}$ and $b, b':C_n(A)\rightarrow C_{n-1}(A)$ and $\lambda:C_n(A)\rightarrow C_n(A)$ be given by \begin{center} $\displaystyle b(a_0\otimes\cdots\otimes a_n)=\sum_{i=0}^{n-1}(-1)^ia_0\otimes\cdots\otimes a_ia_{i+1}\otimes\cdots\otimes a_n+(-1)^na_na_0\otimes\cdots\otimes a_{n-1},$\\
$\displaystyle b'(a_0\otimes\cdots\otimes a_n)=\sum_{i=0}^{n-1}(-1)^ia_0\otimes\cdots\otimes a_ia_{i+1}\otimes\cdots\otimes a_n,$\\ $\lambda(a_0\otimes\cdots\otimes a_n)=(-1)^na_n\otimes a_0\otimes\cdots\otimes a_{n-1}.$\end{center}
It is known that $(1-\lambda)b'=b(1-\lambda)$, so the differential $b$ is well-defined on $C_\bullet^\lambda(A):=C_\bullet(A, A)/Im(1-\lambda)$. The homology of the complex $(C_\bullet^\lambda(A), b)$ is called the cyclic homology of $A$ and is denoted by $HC_\bullet(A)$.

Similar to the case of cyclic cohomology, in characteristic 0, there is a long exact sequence relating the cyclic and Hochschild homology of $A$, $$...\rightarrow HC_n(A){\rightarrow}  HC_{n-2}(A){\rightarrow}  HH_{n-1}(A)\stackrel{I}{\rightarrow}  HC_{n-1}(A)\rightarrow HC_{n-3}(A)\rightarrow ....$$

Part of this paper is to show that similar long exact sequences exist and relate the secondary Hochschild and cyclic (co)homologies associated to a triple $(A, B, \varepsilon)$.

\subsection{The secondary Hochschild cohomology}
Let $A$ be an associative $k$-algebra, $B$  a commutative $k$-algebra, $\varepsilon:B\to A$ a morphism of $k$-algebras such that $\varepsilon(B)\subset {\mathcal Z}(A)$, the center of $A$,  and $M$ an $A$-bimodule symmetric over $B$ (i.e. $\varepsilon(b)m=m\varepsilon(b)$ for all $b\in B$ and $m\in M$).

Given the data above, the secondary Hochschild cohomology of the triple $(A, B, \varepsilon),$ with coefficients in $M,$  was introduced by the second author in \cite{sta} as a tool in the study of algebra deformations, $A[[t]],$ which admit (nontrivial)  $B$-algebra structures. It is denoted by $H^\bullet((A, B, \varepsilon); M)$ and some of its properties were discussed in \cite{s3}. We mention here the existence of a Gerstenhaber algebra structure on $H^\bullet((A, B, \varepsilon); A)$ and of a Hodge-type decomposition, in characteristic zero, for $H^\bullet((A, B, \varepsilon); M)$.

We recall from \cite{sta} the definition of the secondary Hochschild cohomology. The secondary Hochschild complex is defined by setting \begin{center} $C^n((A,B,\varepsilon);M):=Hom_k(A^{\otimes n}\otimes B^{\otimes \frac{n(n-1)}{2}},M).$\end{center} To define the coboundary map, $\delta^{\varepsilon}_n:C^n((A,B,\varepsilon);M)\to C^{n+1}((A,B,\varepsilon);M)$, it is convenient to view an element $T\in A^{\otimes n}\otimes B^{\otimes \frac{n(n-1)}{2}}$ as having the matrix representation
\begin{center}$
T={\otimes}\left(
\begin{array}{cccccccc}
 a_{1}       & b_{1,2} &...&b_{1,n-2}&b_{1,n-1}&b_{1,n}\\
 1  & a_{2} &...&b_{2,n-2}      &b_{2,n-1}&b_{2,n}\\
 .       &. &...&.&.&.\\
 1& 1  &...&1&a_{n-1}&b_{n-1,n}\\
 1& 1  &...&1&1&a_{n}\\
\end{array}
\right),
$\end{center}
where $a_i\in A$, $b_{i,j}\in B,$ and $1\in k$.
For $f\in C^n((A,B,\varepsilon);M)$, we define:
\begin{eqnarray*}&
\delta^{\varepsilon}_n(f)\left(\displaystyle\otimes
\left(
\begin{array}{cccccccc}
 a_{0}& b_{0,1} &  ...&b_{0,n-1}&b_{0,n}\\
1 & a_{1}       & ...&b_{1,n-1}&b_{1,n}\\
. & .    &...&.&.\\
1& 1& ...&a_{n-1}&b_{n-1,n}\\
1 & 1&...&1&a_{n}\\
\end{array}
\right)\right)=\\
& \;&\\
&
a_0\varepsilon(b_{0,1}b_{0,2}...b_{0,n})f\left(\displaystyle\otimes
\left(
\begin{array}{cccccc}
 a_{1}       & b_{1,2} &...&b_{1,n-1}&b_{1,n}\\
1  & a_{2} &...    &b_{2,n-1}&b_{2,n}\\
 .       &. &...&.&.\\
1& 1  &...&a_{n-1}&b_{n-1,n}\\
 1& 1  &...&1&a_{n}\\
\end{array}
\right)\right)+&\\
&\;&\;\\
& \sum_{i=1}^{n}(-1)^if\left(\displaystyle\otimes
\left(
\begin{array}{cccccccc}
 a_{0}       & b_{0,1}&... &b_{0,i-1}b_{0,i}&...&b_{0,n-1}&b_{0,n}\\
1  &a_{1}&...&b_{1,i-1}b_{1,i}     &...&b_{1,n-1}&b_{1,n}\\
 .       &. &...&.&...&.&.\\
  1       &1 &...&a_{i-1}a_i\varepsilon(b_{i-1,i})&...&b_{i-1,n-1}b_{i,n-1}&b_{i-1,n}b_{i,n}\\
   .       &. &...&.&...&.&.\\
1& 1  &...&.&...&a_{n-1}&b_{n-1,n}\\
 1& 1  &...&.&...&1&a_{n}\\
\end{array}
\right)\right)+&\\
& \;&\\
&(-1)^{n+1}f\left(\displaystyle\otimes
\left(
\begin{array}{cccccc}
 a_{0}       & b_{0,1} &...&b_{0,n-2}&b_{0,n-1}\\
1  &a_{1} &...     &b_{1,n-2}&b_{1,n-1}\\
 .       &. &...&.&.\\
1& 1  &...&a_{n-2}&b_{n-2,n-1}\\
 1& 1  &...&1&a_{n-1}\\
\end{array}
\right)\right)a_{n}\varepsilon(b_{0,n}b_{1,n}...b_{n-1,n}).&
\end{eqnarray*}
The following result was proved in \cite{sta}.

\begin{proposition} $(C^n((A,B,\varepsilon);M),\delta_n^{\varepsilon})$ is a complex (i.e. $\delta_{n+1}^{\varepsilon}\delta_n^{\varepsilon}=0$). We denote its homology by $H^n((A,B,\varepsilon);M)$ and we call it the secondary Hochschild cohomology associated to a triple $(A,B,\varepsilon)$ with coefficients in $M$.
\end{proposition}

\subsection{Simplicial category $\Delta$} We recall a few definitions and results from \cite{lo}. Let $\Delta$ be the category whose objects are the finite ordered sets $\overline{n}=\{0,1,...,n\}$, for integers $n\geq 0$, and whose morphisms are the nondecreasing maps. One can show that any morphism in $\Delta$ can be written as the composition of \begin{it}face maps, \end{it} $d^i:\overline{n} \to \overline{n+1},$  and \begin{it} degeneracy maps, \end{it}$s^i:\overline{n} \to \overline{n-1},$ where
$$d^i(u)= \left\{
\begin{array}{c l}
u\; &{\rm if}\;  u< i\\
u+1 &{\rm if}\;  u\geq i,
\end{array}
\right. s^i(u)= \left\{
\begin{array}{c l}
u\; &{\rm if}\;  u\leq i\\
u-1 &{\rm if}\;  u> i.
\end{array}
\right.
$$

The maps $d^i$ and $s^i$ depend on $n$, although the notation does not reflect this dependance. When each composite function below is defined, it can be shown that we have the following identities: \begin{center} $d^{j}d^{i}=d^{i}d^{j-1}$ if $i<j$\end{center}
                                  \begin{center} $s^{j}s^{i}=s^{i}s^{j+1}$ if $i\leq j$\end{center}
                                  \begin{center}  $s^{j}d^i= d^{i}s^{j-1}$ if $i<j$ \end{center}
                                  \begin{center}  $s^{j}d^i=\mathrm{identity}$ if  $i=j, i=j+1$\end{center}
                                  \begin{center} $s^{j}d^i=d^{i-1}s^{j}$ if  $i>j+1.$ \end{center}

A simplicial object $X,$ in a category $\mathcal{C}$, is a functor $X:\Delta^{op}\to \mathcal{C}$. This means that there are objects $X_0, X_1, X_2,\dots$ in $\mathcal{C}$ such that for any nondecreasing map $\phi:\overline{m}\rightarrow\overline{n}$ we have a morphism $X(\phi):X_n\rightarrow X_m$. In addition, for any $\phi$ and $\psi$  nondecreasing maps, we have $X(\phi\circ\psi)=X(\psi)\circ X(\phi)$.

Using the notations $\delta_i:=X(d^i)$ and $\sigma_i:=X(s^i),$ the maps $\delta_i:X_{n+1}\to X_{n}$ and $\sigma_i:X_{n-1}\to X_{n}$ satisfy the following relations.
\begin{eqnarray}
&&\delta_i\delta_j=\delta_{j-1}\delta_i\; {\rm if} \; i<j \label{pres}
\end{eqnarray}
\begin{eqnarray}
\begin{split}
&&\sigma_i\sigma_j=\sigma_{j+1}\sigma_i \; {\rm if} \; i\leq j\\
&&\delta_i\sigma_j=\sigma_{j-1}\delta_i\; {\rm if} \; i<j\\
&&\delta_i\sigma_j=id_{X_n}\; {\rm if} \; i=j\; {\rm or} \; i=j+1\\
&&\delta_i\sigma_j=\sigma_{j}\delta_{i-1}\; {\rm if} \; i>j+1. \label{simp}
\end{split}
\end{eqnarray}
A co-simplicial object in a category $\mathcal{C}$, is a functor $X:\Delta\to \mathcal{C}$.

\section{Simplicial Algebras and Modules}

In this section we introduce (recall) some definitions concerning simplicial algebras. We also give a few examples that will be used extensively in the rest of the paper (the most important being Example \ref{exb}).

\begin{definition} A {\it pre-simplicial} $k$-algebra  $\mathcal{A}$ is a collection of $k$-algebras $A_n$  together with morphisms of $k$-algebras, $\delta_i^{\mathcal{A}}:A_{n}\to A_{n-1}$, such that Equation (\ref{pres}) is satisfied. \end{definition}

\begin{definition} A {\it simplicial $k$-algebra}  $\mathcal{A}$ is a collection of $k$-algebras $A_n$ together with morphisms of $k$-algebras, $\delta_i^{\mathcal{A}}:A_{n}\to A_{n-1}$ and $\sigma_i^{\mathcal{A}}:A_{n}\to A_{n+1}$, such that  Equations (\ref{pres}) and (\ref{simp}) are satisfied.
\end{definition}

\begin{example}\label{presimpl}
Let $A$ be a $k$-algebra. We define the simplicial $k$-algebra $\mathcal{A}(A)$ by setting $A_n=A$, $\delta_i^{\mathcal{A}}=id_A$, and $\sigma_i^{\mathcal{A}}=id_A$.
\end{example}

\begin{example}
Let $A$ be a $k$-algebra, $B$ a commutative $k$-algebra, and $\varepsilon:B\to A$ a morphism of $k$-algebras such that $\varepsilon(B)\subset \mathcal{Z}(A)$. We define the simplicial $k$-algebra $\mathcal{A}(A,B,\varepsilon)$ by setting $A_n=A\otimes B^{\otimes (2n+1)}\otimes A^{op}$ and
\begin{eqnarray*}
&& \; \; \delta_0^{\mathcal{A}}(a\otimes \alpha_1\otimes \alpha_2\otimes ...\otimes \alpha_{n}\otimes \gamma\otimes \beta_1\otimes ...\otimes \beta_n\otimes b)=\\
 && a\varepsilon(\alpha_1)\otimes \alpha_2\otimes ...\otimes \alpha_{n}\otimes \gamma\beta_1\otimes \beta_2 \otimes... \otimes \beta_n\otimes b,\\
&& \; \;  \sigma_0^{\mathcal{A}}(a\otimes \alpha_1\otimes \alpha_2\otimes ...\otimes \alpha_{n}\otimes \gamma\otimes \beta_1\otimes ...\otimes \beta_n\otimes b)=\\
 && a\otimes 1\otimes \alpha_{1}\otimes ...\otimes \alpha_{n}\otimes \gamma\otimes 1\otimes \beta_1\otimes...\otimes \beta_n\otimes b,\\
&& \; \; \delta_i^{\mathcal{A}}(a\otimes \alpha_1\otimes \alpha_2\otimes ...\otimes \alpha_{n}\otimes \gamma\otimes \beta_1\otimes ...\otimes \beta_n\otimes b)=\\
 && a\otimes \alpha_1\otimes ...\otimes \alpha_i\alpha_{i+1}\otimes ...\otimes \gamma\otimes \beta_1\otimes ...\otimes \beta_i\beta_{i+1}\otimes ...\otimes \beta_n\otimes b,\\
 && \; \; \sigma_i^{\mathcal{A}}(a\otimes \alpha_1\otimes \alpha_2\otimes ...\otimes \alpha_{n}\otimes \gamma\otimes \beta_1\otimes ...\otimes \beta_n\otimes b)=\\
 && a\otimes \alpha_1\otimes ...\otimes \alpha_{i}\otimes 1 \otimes \alpha_{i+1}\otimes ...\otimes \gamma\otimes ...\otimes \beta_i\otimes 1\otimes \beta_{i+1}\otimes ...\otimes \beta_n\otimes b,\\
&& \; \; \delta_n^{\mathcal{A}}(a\otimes \alpha_1\otimes \alpha_2\otimes ...\otimes \alpha_{n}\otimes \gamma\otimes \beta_1\otimes ...\otimes \beta_n\otimes b)=\\
 && a\otimes \alpha_1\otimes ...\otimes \alpha_{n-1}\otimes \alpha_{n}\gamma\otimes \beta_1\otimes ...\otimes \beta_{n-1}\otimes \varepsilon(\beta_n)b,\\
&& \; \; \sigma_n^{\mathcal{A}}(a\otimes \alpha_1\otimes \alpha_2\otimes ...\otimes \alpha_{n}\otimes \gamma\otimes \beta_1\otimes ...\otimes \beta_n\otimes b)=\\
&& a\otimes \alpha_1\otimes ...\otimes \alpha_{n}\otimes 1 \otimes \gamma\otimes \beta_1\otimes ...\otimes \beta_{n}\otimes 1\otimes b,
\end{eqnarray*}
for $1\leq i\leq n-1$. Note that for $B=k$ we obtain $\mathcal{A}(A\otimes A^{op})$ as in Example \ref{presimpl}.
\end{example}

\begin{definition}
We say that $\mathcal{M}$ is a {\it pre-simplicial left module} over a pre-simplicial $k$-algebra $\mathcal{A}$ if $\mathcal{M}=(M_n)_{n\geq 0}$ is a pre-simplicial $k$-vector space  together with a left $A_n$-module structure on $M_n$, for each $n$, such that we have the natural compatibility conditions $\delta_i^{\mathcal{M}}(a_nm_n)=\delta_i^{\mathcal{A}}(a_n)\delta_i^{\mathcal{M}}(m_n)$,  for all $a_n\in A_n$ and $m_n\in M_n$.
\end{definition}

\begin{definition}
We say that $\mathcal{M}$ is a {\it simplicial left module} over a simplicial $k$-algebra $\mathcal{A}$ if $\mathcal{M}=(M_n)_{n\geq 0}$ is a simplicial $k$-vector space together with a left $A_n$-module structure on $M_n$, for each $n$, such that we have the natural compatibility conditions $\delta_i^{\mathcal{M}}(a_nm_n)=\delta_i^{\mathcal{A}}(a_n)\delta_i^{\mathcal{M}}(m_n)$ and $\sigma_i^{\mathcal{M}}(a_nm_n)=\sigma_i^{\mathcal{A}}(a_n)\sigma_i^{\mathcal{M}}(m_n)$, for all $a_n\in A_n$ and $m_n\in M_n$.
\end{definition}

\begin{remark} The above definitions can be easily adapted to (pre)simplicial right modules or bimodules. \end{remark}

\begin{example} If $A$ is a $k$-algebra and $M$ is a left $A$-module, we can construct the following (pre)simplicial left modules over $\mathcal{A}(A)$.

a) For each nonnegative integer $p$, let $\mathcal{M}(M(p))$ be the pre-simplicial module obtained by setting $M_p=M,$ $M_n=0$ if $n\neq p$, and  $\delta_i^{\mathcal{M}}=0$.  For $a_p\in A_p=A$ and $m_p\in M_p=M,$ the product $a_pm_p$ is given by the left $A$-module structure on $M$.

b) The simplicial left module $\mathcal{M}(M)$ is defined by setting  $M_n=M,$ $\delta_i^{\mathcal{M}}=id_M$, and $\sigma_i^{\mathcal{M}}=id_M$. For $a_n\in A_n=A$ and $m_n\in M_n=M,$ the product $a_nm_n$ is given by the left $A$-module structure on $M$.

c) If $M$ is an $A$-bimodule then $\mathcal{M}(M(p))$ and $\mathcal{M}(M)$ are (pre-)simplicial left modules over $\mathcal{A}(A\otimes A^{op})$
\end{example}

\begin{example}\label{rebar} If $A$ is a $k$-algebra then the bar resolution $\mathcal{B}(A)$, is a simplicial left module over $\mathcal{A}(A\otimes A^{op})$. We recall that $B_n=A^{\otimes (n+2)}$, the $A\otimes A^{op}$-module structure is given by
\begin{eqnarray*}
(a\otimes b)\cdot (a_0\otimes a_1\otimes ...\otimes a_{n+1})&=& aa_0\otimes a_1\otimes ...\otimes a_{n+1}b,
\end{eqnarray*}
and for $0\leq i\leq n$ we have
\begin{eqnarray*}\label{rebar}
\delta_i^{\mathcal{B}}(a_0\otimes a_1\otimes...\otimes a_{n+1})&=&a_0\otimes ...\otimes a_{i}a_{i+1}\otimes ...\otimes a_{n+1},\\
\sigma_i^{\mathcal{B}}(a_0\otimes a_1\otimes...\otimes a_{n+1})&=&a_0\otimes ...\otimes a_{i}\otimes 1\otimes a_{i+1}\otimes ...\otimes a_{n+1}.
\end{eqnarray*}
\end{example}
The next three examples arise from the context in which the secondary Hochschild cohomology is defined. Let $A$ be a $k$-algebra, $B$ a commutative $k$-algebra, $\varepsilon:B\to A$ a morphism of $k$-algebras such that $\varepsilon(B)\subset \mathcal{Z}(A)$, and $M$ an $A$-bimodule which is symmetric over $B$.

\begin{example}
 a) For each nonnegative integer $p$, we define the pre-simplicial left module $\mathcal{M}(M(p))$, over the simplicial $k$-algebra $\mathcal{A}(A,B,\varepsilon)$, by setting $M_p=M$, $M_n=0$ for $n\neq p$, $\delta_i=0$.  The only nontrivial multiplication is $A_p\times M_p\to M_p$  given by the formula
\begin{eqnarray}
&(a\otimes \alpha_1\otimes ...\otimes \alpha_p\otimes \gamma\otimes \beta_1\otimes...\otimes \beta_p\otimes b)\cdot m_p=a\varepsilon(\alpha_1...\alpha_{p}\gamma\beta_1...\beta_p)m_pb.&\label{bimodstructure1}
\end{eqnarray}
b) The simplicial bimodule $\mathcal{M}(M)$, over the simplicial $k$-algebra $\mathcal{A}(A,B,\varepsilon)$, is defined by taking $M_n=M$, $\delta_i=id_M$, and $\sigma_i=id_M$. The multiplication $A_n\times M_n\to M_n$  is given by the formula
\begin{eqnarray}
&(a\otimes \alpha_1\otimes ...\otimes \alpha_n\otimes \gamma\otimes \beta_1\otimes...\otimes \beta_n\otimes b)\cdot m_n=a\varepsilon(\alpha_1...\alpha_{n}\gamma\beta_1...\beta_n)m_nb.&
\end{eqnarray}
\end{example}

Our next example can be viewed as an analogue of the Hochschild bar resolution in the context of secondary Hochschild cohomology. This example will be used extensively in the next sections.
\begin{example}
The bar simplicial left module $\mathcal{B}(A,B,\varepsilon)$, over the simplicial $k$-algebra $\mathcal{A}(A,B,\varepsilon)$, is defined by setting $B_n=A^{\otimes (n+2)}\otimes B^{\otimes \frac{(n+1)(n+2)}{2}}$. For $0\leq i\leq n$ we define
\begin{eqnarray*}
&\delta_i^{\mathcal{B}}\left(\displaystyle\otimes
\left(
\begin{array}{cccccccc}
 a_{0}& b_{0,1} &  ...&b_{0,n+1}\\
1 & a_{1}    &...&b_{1,n+1}\\
. & .       &...&.\\
1& 1& ...&b_{n,n+1}\\
1 & 1& ...&a_{n+1}\\
\end{array}
\right)\right)=&\\
&\otimes
\left(
\begin{array}{cccccccc}
 a_{0}& b_{0,1} &  ...&b_{0,i}b_{0,i+1}&...&b_{0,n}&b_{0,n+1}\\
1 & a_{1}    &...&b_{1,i}b_{1,i+1}&...&b_{1,n}&b_{1,n+1}\\
. & .       &...&.&...&.&.\\
. & .       &...&a_ia_{i+1}\varepsilon(b_{i,i+1})&...&b_{i,n}b_{i+1,n}&b_{i,n+1}b_{i+1,n+1}\\
. & .       &...&.&...&.&.\\
1& 1 &...&.&...&a_n&b_{n,n+1}\\
1 & 1&...&.&...&1&a_{n+1}\\
\end{array}
\right),&\\
&\sigma_i^{\mathcal{B}}\left(\displaystyle\otimes
\left(
\begin{array}{cccccccc}
 a_{0}& b_{0,1} &  ...&b_{0,n+1}\\
1 & a_{1}    &...&b_{1,n+1}\\
. & .       &...&.\\
1& 1& ...&b_{n,n+1}\\
1 & 1& ...&a_{n+1}\\
\end{array}
\right)\right)=&\\
&
\otimes
\left(
\begin{array}{cccccccccc}
 a_{0}& b_{0,1} &  ...&b_{0,i}&1&b_{0,i+1}&...&b_{0,n}&b_{0,n+1}\\
1 & a_{1}    &...&b_{1,i}&1&b_{1,i+1}&...&b_{1,n}&b_{1,n+1}\\
. & .       &...&.&.&.&...&.&.\\
. & .       &...&a_{i}&1&b_{i,i+1}&...&b_{i,n}&b_{i,n+1}\\
1 & 1       &...&1&1&1&...&1&1\\
. & .       &...&1&1&a_{i+1}&...&b_{i+1,n}&b_{i+1,n+1}\\
. & .       &...&.&.&.&...&.&.\\
1& 1 &...&.&1&.&...&a_n&b_{n,n+1}\\
1 & 1&...&.&1&.&...&1&a_{n+1}\\
\end{array}
\right).
&
\end{eqnarray*}
Finally, the multiplication on $B_n$ is given by
\begin{eqnarray*}
&(a\otimes \alpha_1\otimes ...\otimes \alpha_{n}\otimes \gamma\otimes \beta_1\otimes ...\otimes \beta_n\otimes b)\left(\displaystyle\otimes
\left(
\begin{array}{cccccccc}
 a_{0}& b_{0,1} &  ...&b_{0,n+1}\\
1 & a_{1}    &...&b_{1,n+1}\\
. & .       &...&.\\
1& 1& ...&b_{n,n+1}\\
1 & 1& ...&a_{n+1}\\
\end{array}
\right)\right)=&\\
&
\otimes
\left(
\begin{array}{cccccccc}
a a_{0}& \alpha_{1}b_{0,1} &  ...&\alpha_nb_{0,n}&\gamma b_{0,n+1}\\
1 & a_{1}    &...&b_{1,n}&b_{1,n+1}\beta_1\\
. & .       &...&.&.\\
1& 1& ...&a_{n}&b_{n,n+1}\beta_n\\
1 & 1& ...&1&a_{n+1}b\\
\end{array}
\right).&
\end{eqnarray*}
 \label{exb}
\end{example}

\begin{remark} The bar construction described above can be formulated in terms of bimodules. This is the analog of the fact that an $A$-bimodule is  a left $A\otimes A^{op}$-module.  Our choice to use left modules turns out to be more convenient in the next section when we have to consider $Hom$ and $\otimes$ of simplicial modules over a simplicial algebra $\mathcal{A}$.
\end{remark}

\section{Secondary Cyclic Cohomology}

\subsection{Secondary cohomology of a triple $(A,B,\varepsilon)$}
In the arXiv version of \cite{sta} we introduced a certain cyclic module associated to a triple $(A,B,\varepsilon)$. That construction does not have all the properties which one would like from a cyclic cohomology, so the construction  was removed from the final version of the paper. In this section we take a different approach to the secondary cyclic cohomology, one which seems to be more consistent with the classical theory and  could have applications to geometry.

We will use the bar simplicial module $\mathcal{B}(A,B,\varepsilon)$ to  introduce the analogue of the Hochschild complex of $A$, with coefficients in $A^*=Hom_k(A, k)$, and then use it to define the secondary cyclic cohomology.

\begin{definition}
We say that $\mathcal{M}$ is a {\it co-simplicial left module} over a simplicial $k$-algebra $\mathcal{A}$ if $\mathcal{M}=(M^n)_{n\geq 0}$ is a co-simplicial $k$-vector space together with a left $A_n$-module structure on $M^n$, for each $n$, such that we have the natural compatibility conditions $\delta^i_{\mathcal{M}}(\delta_i^{\mathcal{A}}(a_{n+1})m_{n})=a_{n+1}\delta^i_{\mathcal{M}}(m_n)$ and $\sigma^i_{\mathcal{M}}(\sigma_i^{\mathcal{A}}(a_{n-1})m_n)=a_{n-1}\sigma^i_{\mathcal{M}}(m_n)$ for all $a_{n-1}\in A_{n-1}$, $a_{n+1}\in A_{n+1}$ and $m_n\in M^n$.
\end{definition}

For a simplicial left module  $(\mathcal{X}, \delta_i^\mathcal{X}, \sigma_i^\mathcal{X})$ and  a co-simplicial left module $(\mathcal{Y}, \delta_\mathcal{Y}^i, \sigma_\mathcal{Y}^i)$, both over the simplicial algebra $\mathcal{A}$, we define $\mathcal{M}^n:=Hom_{A_n}(X_n, Y^n)$ and the maps $D^i: \mathcal{M}^n\rightarrow \mathcal{M}^{n+1}$ and $S^i:\mathcal{M}^n\rightarrow \mathcal{M}^{n-1}$ by $$D^i(f)=\delta^i_{\mathcal{Y}}f\delta_i^{\mathcal{X}},$$
$$S^i(f)=\sigma^i_{\mathcal{Y}}f\sigma_i^{\mathcal{X}}.$$

With the above definitions we have the following lemma.
\begin{lemma} $(\mathcal{M}, D^i, S^i)$ is a co-simplicial $k$-module, which we will denote by $Hom_{\mathcal{A}}(\mathcal{X},\mathcal{Y})$. \label{lemma3}
\end{lemma}
\begin{proof}
An easy verification of the definition.
\end{proof}
\begin{example} We define a co-simplicial module $\mathcal{H}(A,B,\varepsilon)$ over $\mathcal{A}(A,B,\varepsilon)$ as follows. First, take $\mathcal{H}^n=Hom_k(A\otimes B^{\otimes n},k)$, with the $A_n$-left module structure given by
\begin{eqnarray*}
((a\otimes \alpha_1\otimes ...\otimes \alpha_{n}\otimes \gamma\otimes \beta_1\otimes ...\otimes \beta_n\otimes b)\cdot\phi) (a_0\otimes b_1\otimes ...\otimes b_n)=\\
\phi(ba_0a\varepsilon(\gamma)\otimes \alpha_1b_1\beta_1\otimes ...\otimes \alpha_nb_n\beta_n)
\end{eqnarray*}
Second, we define $\delta^i_{\mathcal{H}}:\mathcal{H}^n\to \mathcal{H}^{n+1}$  by
\begin{eqnarray*}
\delta^0_{\mathcal{H}}(\phi)(a\otimes b_1\otimes ...\otimes b_{n+1})&=&\phi(a\varepsilon(b_1)\otimes b_2\otimes ...\otimes b_{n+1}),\\
\delta^{n+1}_{\mathcal{H}}(\phi)(a\otimes b_1\otimes ...\otimes b_{n+1})&=&\phi(\varepsilon(b_{n+1})a\otimes b_1\otimes ...\otimes b_{n}),\\
\delta^i_{\mathcal{H}}(\phi)(a\otimes b_1\otimes ...\otimes b_{n+1})&=&\phi(a\otimes b_1\otimes ...\otimes b_ib_{i+1}\otimes...\otimes b_{n+1}),
\end{eqnarray*}
if $1\leq i\leq n$. Finally, we set $\sigma^i_{\mathcal{H}}:\mathcal{H}^n\to \mathcal{H}^{n-1}$ to be given by
\begin{eqnarray*}
\sigma^i_{\mathcal{H}}(\phi)(a\otimes b_1\otimes ...\otimes b_{n-1})=\phi(a\otimes b_1\otimes...\otimes b_{i}\otimes 1\otimes b_{i+1}\otimes  ...\otimes b_{n-1}),
\end{eqnarray*}
if $0\leq i\leq n-1$. \label{exh}
\end{example}
\begin{proof}
Straightforward computation.
\end{proof}

Now, if we combine the bar simplicial module $\mathcal{B}(A,B, \varepsilon)$ (Example \ref{exb}), the co-simplicial module $\mathcal{H}(A,B,\varepsilon)$ (Example \ref{exh}), and Lemma \ref{lemma3}, we get the following result.

\begin{proposition} $Hom_{\mathcal{A}(A,B,\varepsilon)}(\mathcal{B}(A,B,\varepsilon), \mathcal{H}(A,B,\varepsilon))$ is a co-simplicial $k$-module.  We denote the associated cohomology groups  by $HH^n(A,B,\varepsilon)$,  and we  call them the secondary Hochschild cohomology groups associated to the triple $(A,B,\varepsilon)$. \label{prop2}
\end{proposition}
\begin{proof}
Follows from the above discussion.
\end{proof}

Next we will study in more details the  complex $C^n(A, B, \varepsilon)$, which defines the secondary Hochschild cohomology $HH^n(A,B,\varepsilon)$.
Notice that
$$C^n(A, B, \varepsilon)=Hom_{A_n}(B_n(A,B,\varepsilon), H^n(A,B,\varepsilon)).$$
From Proposition \ref{prop2} we have $d^i:C^n(A, B, \varepsilon)\rightarrow C^{n+1}(A, B, \varepsilon)$, defined by $d^i(f)=\delta^i_{\mathcal{H}}f\delta_i^{\mathcal{B}}$. More precisely, we have:

\begin{eqnarray*}
&d^{0}(f)\left(\displaystyle\otimes
\left(
\begin{array}{cccccccc}
 a_{0}& b_{01} &  ...&b_{0,n+2}\\
1 & a_{1}    &...&b_{1,n+2}\\
. & .       &...&.\\
1& 1& ...&b_{n+1,n+2}\\
1 & 1& ...&a_{n+2}\\
\end{array}
\right)\right)\left(
                \begin{array}{c}
                  a \\
                  \alpha_{1} \\
                  ...\\
                  \alpha_{n} \\
                  \alpha_{n+1} \\
                \end{array}
              \right)
=&\\
&f\left(\displaystyle\otimes
\left(
\begin{array}{cccccccc}
 a_{0}a_1\varepsilon(b_{0,1})& b_{0,2}b_{1,2} &  ...&b_{0,n+2}b_{1,n+2}\\
1 & a_{2}    &...&b_{1,n+2}\\
. & .       &...&.\\
. & .      &...&b_{n-1,n+2}\\
1 & 1&...&a_{n+2}\\
\end{array}
\right)\right)\left(
                \begin{array}{c}
                   a\varepsilon(\alpha_{1})\\
                  \alpha_{2} \\
                  . \\
									\alpha_{n}\\
                  \alpha_{n+1} \\
                \end{array}
              \right),&\\&d^i(f)\left(\displaystyle\otimes
\left(
\begin{array}{cccccccc}
 a_{0}& b_{01} &  ...&b_{0,n+2}\\
1 & a_{1}    &...&b_{1,n+2}\\
. & .       &...&.\\
1& 1& ...&b_{n+1,n+2}\\
1 & 1& ...&a_{n+2}\\
\end{array}
\right)\right)\left(
                \begin{array}{c}
                  a \\
                  \alpha_{1} \\
                  ...\\
                  \alpha_{n} \\
                  \alpha_{n+1} \\
                \end{array}
              \right)
=&\\
&f\left(\displaystyle\otimes
\left(
\begin{array}{cccccccc}
 a_{0}& b_{0,1} &  ...&b_{0,i}b_{0,i+1}&...&b_{0,n+2}\\
1 & a_{1}    &...&b_{1,i}b_{1,i+1}&...&b_{1,n+2}\\
. & .       &...&.&...&.\\
. & .       &...&a_ia_{i+1}\varepsilon(b_{i,i+1})&...&b_{i,n+2}b_{i+1,n+2}\\
. & .       &...&.&...&.\\
1 & 1&...&.&...&a_{n+2}\\
\end{array}
\right)\right)\left(
                \begin{array}{c}
                  a \\
                  \alpha_{1} \\
                  . \\
									\alpha_i\alpha_{i+1}\\
									.\\
                  \alpha_{n+1} \\
                \end{array}
              \right),&
\end{eqnarray*}
 for $1\leq i\leq n$, and
 \begin{eqnarray*}
&d^{n+1}(f)\left(\displaystyle\otimes
\left(
\begin{array}{cccccccc}
 a_{0}& b_{01} &  ...&b_{0,n+2}\\
1 & a_{1}    &...&b_{1,n+2}\\
. & .       &...&.\\
1& 1& ...&b_{n+1,n+2}\\
1 & 1& ...&a_{n+2}\\
\end{array}
\right)\right)\left(
                \begin{array}{c}
                  a \\
                  \alpha_{1} \\
                  ...\\
                  \alpha_{n} \\
                  \alpha_{n+1} \\
                \end{array}
              \right)
=&\\\end{eqnarray*} \begin{eqnarray*}
&f\left(\displaystyle\otimes
\left(
\begin{array}{cccccccc}
 a_{0}& b_{0,1} &  ...&b_{0,n+1}b_{0,n+2}\\
1 & a_{1}    &...&b_{1,n+1}b_{1,n+2}\\
. & .       &...&.\\
. & .      &...&b_{n,n+1}b_{n,n+2}\\
1 & 1&...&a_{n+1}a_{n+2}\varepsilon(b_{n+1,n+2})\\
\end{array}
\right)\right)\left(
                \begin{array}{c}
                   \varepsilon(\alpha_{n+1})a\\
                  \alpha_{1} \\
                  . \\
									\alpha_{n-1}\\
                  \alpha_{n} \\
                \end{array}
              \right).&
\end{eqnarray*}

The differential $\partial_n:C^n(A, B, \varepsilon)\rightarrow C^{n+1}(A, B, \varepsilon)$ is given by the formula $\partial_{n}=\sum_{i=0}^{n+1}(-1)^{i}d^i$.

Using the standard properties of adjoint functors we can identify $C^n(A, B, \varepsilon)$ with $Hom_k(A^{\otimes n}\otimes B^{\otimes \frac{n(n-1)}{2}}, Hom_k(A\otimes B^{\otimes n},k))$ or, even better, with $Hom_k(A^{\otimes (n+1)}\otimes B^{\otimes \frac{n(n+1)}{2}},k)$.
The identification is given by
\begin{eqnarray*}
\Psi_n:C^n(A,B,\varepsilon)\cong Hom_k(A^{\otimes (n+1)}\otimes B^{\otimes \frac{n(n+1)}{2}},k),
\end{eqnarray*}
\begin{eqnarray*}
&\Psi_n(f)\left(\displaystyle\otimes
\left(
\begin{array}{cccccccc}
 a_{0}& b_{01} &  ...&b_{0,n}\\
1 & a_{1}    &...&b_{1,n}\\
. & .       &...&.\\
1& 1& ...&b_{n-1,n}\\
1 & 1& ...&a_{n}\\
\end{array}
\right)\right)
=&\\
&f\left(\displaystyle\otimes
\left(
\begin{array}{cccccccc}
 1& 1 &  ...&1&1\\
1 & a_{1}    &...&b_{1,n+1}&1\\
. & .       &...&.&.\\
. & .       &...&b_{n-1,n}&1\\
. & .      &...&a_{n}&1\\
1 & 1&...&1&1\\
\end{array}
\right)\right)\left(
                \begin{array}{c}
                   a_0\\
                  b_{0,1} \\
                  . \\
									b_{0,n-1}\\
                  b_{0,n} \\
                \end{array}
              \right).&
\end{eqnarray*}

Combining all the results from this section, one can see that the induced differential, $b_n:C^n(A,B,\varepsilon)\to C^{n+1}(A,B,\varepsilon)$, has the following formula:
\begin{eqnarray*}
&(b_n\phi)\left(\displaystyle\otimes
\left(
\begin{array}{cccccccc}
 a_{0}& b_{0,1} &  ...&b_{0,n}&b_{0,n+1}\\
1 & a_{1}    &...&b_{1,n}&b_{1,n+1}\\
. & .       &...&.&.\\
1& 1& ...&a_n&b_{n,n+1}\\
1 & 1& ...&1&a_{n+1}\\
\end{array}
\right)\right)=&\\
&\sum_{i=0}^{n}(-1)^i\phi\left(\displaystyle\otimes
\left(
\begin{array}{cccccccc}
 a_{0}&...& b_{0,i}b_{0,i+1} &  ...&b_{0,n}&b_{0,n+1}\\
1 & ...&b_{1,i}b_{1,i+1}    &...&b_{1,n}&b_{1,n+1}\\
. &...& .       &...&.&.\\
1& ...&a_ia_{i+1}\varepsilon(b_{i,i+1})& ...&b_{i,n}b_{i+1,n}&b_{i,n+1}b_{i+1,n+1}\\
. &...& .       &...&.&.\\
1 & ...&1& ...&1&a_{n+1}\\
\end{array}
\right)\right)+&\\
&(-1)^{n+1}\phi\left(\displaystyle\otimes
\left(
\begin{array}{cccccccc} a_{n+1}a_{0}\varepsilon(b_{0,n+1})&b_{1,n+1}b_{0,1}&...& b_{i,n+1}b_{0,i} &  ...&b_{n,n+1}b_{0,n}\\
1 & a_1&...&b_{1,i}    &...&b_{1,n}\\
. &.&...& .       &...&.\\
1& .&...&a_i& ...&b_{i,n}\\
. &.&...& .       &...&.\\
1 & .&...&1& ...&a_{n}\\
\end{array}
\right)\right).&
\end{eqnarray*}

\begin{remark}
We just proved that the secondary Hochschild cohomology associated to the triple $(A,B,\varepsilon)$ is the homology of the complex $(C^n(A, B, \varepsilon), b_n)$. When $B=k$, we have that $\mathcal{H}^n=A^*$ and we recover the known fact that $HH^n(A)$ is the cohomology of the Hochschild complex with coefficients in $A^*$.
\end{remark}

\begin{remark} Let $M$ be an $A$-bimodule that is $B$-symmetric. We denote by $\mathcal{C}(M)$ the co-simplicial $\mathcal{A}(A,B, \varepsilon)$-module determined by $\mathcal{C}^n(M)=M$ with the obvious $A\otimes B^{\otimes (2n+1)}\otimes A^{op}$ left module structure,  $\delta_{\mathcal{C}}^i=id_M$ and $\sigma_{\mathcal{C}}^i=id_M$. Then from Lemma \ref{lemma3} we know that $Hom_{\mathcal{A}(A,B,\varepsilon)}(\mathcal{B}(A,B,\varepsilon), \mathcal{C}(M))$ is a co-simplicial $k$-module. One can see that the homology of the induced complex is the secondary Hochschild cohomology $H^{\bullet}((A,B,\varepsilon); M)$.
\end{remark}

\begin{remark}
Although similar in name, the reader should note the difference between $HH^\bullet(A,B,\varepsilon)$, the secondary Hochschild cohomology associated to the triple $(A, B, \varepsilon)$, and $H^\bullet((A, B,\varepsilon), -)),$ the secondary Hochschild cohomology of the triple $(A, B, \varepsilon)$. In general, the former is not a special case of the latter since the ``coefficient" module $\mathcal{H}^n$ varies with $n$.
\end{remark}

\subsection{Secondary Cyclic Cohomology}
There is a natural left action of the cyclic group $C_{n+1}=<\lambda>$ on $C^n(A,B,\varepsilon)$ (where $\lambda$  is the permutation $(0, 1, 2,\dots, n)$),  defined as follows:
\begin{eqnarray*}
&\lambda\cdot\phi\left(\displaystyle\otimes
\left(
\begin{array}{cccccccc}
 a_{0}& b_{0,1} &b_{0,2}&  ...&b_{0,n-1}&b_{0,n}\\
1 & a_{1}    &b_{1,2}&...&b_{1,n-1}&b_{1,n}\\
. & .       &.&...&.&.\\
1& 1& 1&...&a_{n-1}&b_{n-1,n}\\
1 & 1& 1&...&1&a_{n}\\
\end{array}
\right)\right)=&\\
&(-1)^{n}\phi\left(\displaystyle\otimes
\left(
\begin{array}{cccccccc}
 a_{n}& b_{0,n} &b_{1,n}&  ...&b_{n-2,n}&b_{n-1,n}\\
1 & a_{0}    &b_{0,1}&...&b_{0,n-2}&b_{0,n-1}\\
. & .       &.&...&.&.\\
1& 1& 1&...&a_{n-2}&b_{n-2,n-1}\\
1 & 1& 1&...&1&a_{n-1}\\
\end{array}
\right)\right).&
\end{eqnarray*}

Similarly, we have a right action of $C_{n+1}$ on $C^n(A, B, \varepsilon)$ ($\phi\lambda=\lambda^n\phi)$. We define the secondary cyclic complex as a sub-complex of $C^\bullet(A, B, \varepsilon)$.  A multilinear functional $\phi\in C^n(A, B, \varepsilon)$ is called a secondary cyclic $n$-cochain if it is invariant under the action of $C_{n+1}$. That is, $\lambda\phi=\phi$ (a condition equivalent to $\phi\lambda=\phi$).

The $k$-submodule of $C^n(A, B, \varepsilon)$ of secondary cyclic $n$-cochains is denoted by $C^n_{\lambda}(A,B,\varepsilon)$.
Note that the action of $\lambda$ determines an operator on $C^n(A,B,\varepsilon)$ and we have that $C^n_{\lambda}(A,B,\varepsilon)=\mathrm{Ker}(1-\lambda)$.

\begin{lemma} For every $n\geq 0$ we have $b_n(C^n_{\lambda}(A,B,\varepsilon))\subseteq C^{n+1}_{\lambda}(A,B,\varepsilon)$.
\end{lemma}
\begin{proof}  Similar to the case of cyclic cochains, one can show that  $(1-\lambda)b_n=b_n'(1-\lambda)$, where $b_n'$ is the sum of the first $n+1$ terms from the definition of $b_n$. Since $C^n_{\lambda}(A,B,\varepsilon)=\mathrm{Ker}(1-\lambda)$, we obtain the stated result.

\end{proof}

\begin{definition} The homology of the complex $(C^\bullet_{\lambda}(A,B,\varepsilon) , b)$ is denoted by $HC^\bullet(A,B,\varepsilon)$
and is called the \begin{it}Secondary Cyclic Cohomology \end{it}associated to the triple $(A,B,\varepsilon).$
\end{definition}

Next we obtain a similar result to Connes' long exact sequence.

\begin{theorem}\label{Connes} Let $k$ be a field of characteristic zero. For a triple $(A,B,\varepsilon)$ we have a long exact sequence
$$...\rightarrow HC^n(A,B,\varepsilon)\stackrel{I}{\rightarrow}  HH^n(A,B,\varepsilon)\stackrel{B}{\rightarrow}  HC^{n-1}(A,B,\varepsilon)\stackrel{S}{\rightarrow}  HC^{n+1}(A,B,\varepsilon)\rightarrow ...,$$
where $I$ is induced by the inclusion $C^{n}_{\lambda}(A,B,\varepsilon)\to C^{n}(A,B,\varepsilon)$.
\end{theorem}
\begin{proof}
We will follow the line of proof from \cite{k} and \cite{lo}. Note that the short exact sequence of complexes
$$0\to C^{\bullet}_{\lambda}(A,B,\varepsilon)\to C^{\bullet}(A,B,\varepsilon)\to \frac{C^{\bullet}(A,B,\varepsilon)}{C^{\bullet}_{\lambda}(A,B,\varepsilon)}\to 0$$
induces a long exact sequence
\begin{eqnarray*}
...\to HC^{n}(A,B,\varepsilon)\stackrel{I}{\rightarrow} HH^{n}(A,B,\varepsilon)\to H^{n}(\frac{C^{\bullet}(A,B,\varepsilon)}{C^{\bullet}_{\lambda}(A,B,\varepsilon)})\to HC^{n+1}(A,B,\varepsilon)\to....
\end{eqnarray*}

In order to get the statement from the theorem we need to show that \begin{center} $HC^{n-1}(A,B,\varepsilon)\simeq H^n(C^{\bullet}(A,B,\varepsilon)/C^{\bullet}_{\lambda}(A,B,\varepsilon))$.\end{center}

Take $N=1+\lambda+...+\lambda^n:C^n(A,B,\varepsilon)\to C^n(A,B,\varepsilon)$. One can check that $(1-\lambda)b=b'(1-\lambda)$, $N(1-\lambda)=(1-\lambda)N=0,$ and $bN=Nb'$, so we obtain the short exact sequence of complexes:
$$0\to \frac{C^{\bullet}(A,B,\varepsilon)}{C^{\bullet}_{\lambda}(A,B,\varepsilon)} \stackrel{1-\lambda}{\rightarrow} (C^{\bullet}(A,B,\varepsilon),b')\stackrel{N}{\rightarrow} C^{\bullet}_{\lambda}(A,B,\varepsilon)\to 0.$$
The only part which requires an additional explanation  is the inclusion  $\mathrm{Ker}(N)\subset \mathrm{Im}(1-\lambda).$ If $\phi\in\mathrm{Ker}(N)$, then let $\psi=\phi-\lambda^2\phi-2\lambda^3\phi-\dots-(n-1)\lambda^n\phi$. It is then easy to see that we have $(1-\lambda)\left(\frac{1}{n+1}\psi\right)=\phi$.

Finally, the complex $(C^{\bullet}(A,B,\varepsilon),b')$ has the contracting homotopy \begin{center} $s_nf\left(\displaystyle\otimes
\left(
\begin{array}{cccccccc}
 a_{0}& b_{0,1} &  ...&b_{0,n-1}\\
1 & a_{1}       & ...&b_{1,n-1}\\
. & .       &...&.\\
1 & 1& ...&a_{n-1}\\
\end{array}
\right)\right)=(-1)^{n-1}f\left(\displaystyle\otimes
\left(
\begin{array}{cccccccc}
 a_{0}& b_{0,1} &  ...&b_{0,n-1}&1\\
1 & a_{1}       & ...&b_{1,n-1}&1\\
. & .       &...&.&1\\
1 & 1& ...&a_{n-1}&1\\
1 & 1& ... &1& 1\\
\end{array}
\right)\right).$\end{center}
Thus, since $H^n(C^{\bullet}(A,B,\varepsilon),b')=0$, we get an isomorphism
$HC^{n-1}(A,B,\varepsilon)\simeq H^n(C^{\bullet}(A,B,\varepsilon)/C^{\bullet}_{\lambda}(A,B,\varepsilon))$ and our proof is now complete.
\end{proof}

\begin{remark} Theorem \ref{Connes} is a consequence of the fact that $C^{n}_{\lambda}(A,B,\varepsilon)$ is a cyclic object.
\end{remark}

\section{Secondary Homology}
In this section we give a brief description for the secondary (cyclic) homology associated to a triple $(A,B,\varepsilon)$.

We will use the bar simplicial  module $\mathcal{B}(A,B,\varepsilon)$ over the simplicial algebra $\mathcal{A}(A,B,\varepsilon)$.
The first step is to prove an analog of the Lemma \ref{lemma3} replacing the $Hom_{\mathcal{A}}$ functor with the tensor product $\otimes_{\mathcal{A}}$.

Let $(\mathcal{X}, \delta_i^\mathcal{X}, \sigma_i^\mathcal{X})$ be a right simplicial module and $(\mathcal{Y},\delta_i^{\mathcal{Y}}, \sigma_i^{\mathcal{Y}})$ a left simplicial module, both over a simplicial algebra $\mathcal{A}$. We set $\mathcal{M}_n:=X_n\otimes_{A_n}Y_n$ and define the maps $D_i:\mathcal{M}_n\to\mathcal{M}_{n-1}$  and $S_i:\mathcal{M}_n\to \mathcal{M}_{n+1}$ by the formulas
$$D_i=\delta_i^{\mathcal{X}}\otimes\delta_i^{\mathcal{Y}}:X_n\otimes_{A_n}Y_n\to X_{n-1}\otimes_{A_{n-1}}Y_{n-1}$$
$$S_i=\sigma_i^{\mathcal{X}}\otimes\sigma_i^{\mathcal{Y}}:X_n\otimes_{A_n}Y_n\to X_{n+1}\otimes_{A_{n+1}}Y_{n+1}.$$
One can check that these maps are well defined. Moreover, we have the following lemma.

\begin{lemma} \label{lemma4}$(\mathcal{M}, D_i, S_i)$ is a simplicial $k$-module, which we'll denote by $\mathcal{X}\otimes_{\mathcal{A}} \mathcal{Y}$.
\end{lemma}
\begin{proof} An easy verification of the definitions.
\end{proof}
The next two examples, obtained by specializing the above construction, lead to the definition of the secondary Hochschild homology and the secondary cyclic homology.

\begin{example} Let $M$ be an $A$-bimodule that is $B$-symmetric. We denote with $\mathcal{S}(M)$ the simplicial $\mathcal{A}(A,B, \varepsilon)$-module determined by $S_n:=M$ with the obvious $A\otimes B^{\otimes (2n+1)}\otimes A^{op}$-right module structure,  $\delta^{\mathcal{S}}_i=id_M$ and $\sigma^{\mathcal{S}}_i=id_M$. Then, from Lemma \ref{lemma4} we know that $\mathcal{S}(M)\otimes_{\mathcal{A}(A,B,\varepsilon)}\mathcal{B}(A,B,\varepsilon)$ is a simplicial $k$-module. \label{ex5}
\end{example}

\begin{definition}
The homology of the complex associated to the simplicial $k$-module from Example \ref{ex5} is called the secondary Hochschild homology of the triple $(A,B,\varepsilon)$ with coefficients in $M$ and it is denoted by  $H_{\bullet}((A,B,\varepsilon), M)$.
\end{definition}

\begin{remark} If we make the identification $S_n\otimes_{A_n}B_n=M\otimes A^{\otimes n}\otimes B^{\otimes \frac{n(n-1)}{2}}$, one can easily see that differential is given by the formula

\begin{eqnarray*}
&\partial_n\left(\displaystyle m\otimes
\left(
\begin{array}{cccccccc}
 a_{0}& b_{0,1} & ...& b_{0,i} & ...&b_{0,n-1}\\
1 & a_{1} &...&b_{1,i}   &...&b_{1,n-1}\\
. & . &...&.      &...&.&\\
1& 1&...&a_{i}& ...&b_{i,n-1}\\
. & . &...&.      &...&.&\\
1 & 1& ...&.&...&a_{n-1}\\
\end{array}
\right)\right)=&\\& \;&\\
&
ma_0\varepsilon(b_{0,1}b_{0,2}...b_{0,n-1})\left(\displaystyle\otimes
\left(
\begin{array}{cccccc}
 a_{1}       & b_{1,2} &...&b_{1,i}&...&b_{1,n-1}\\
1  & a_{2} &...    &b_{2,i}&...&b_{2,n-1}\\
 .       &. &...&.\\
1& 1&...&a_{i}& ...&b_{i,n-1}\\
  .       &. &...&.\\
 1& 1  &...&1&...&a_{n-1}\\
\end{array}
\right)\right)+&\\
&\;&\;\\ \end{eqnarray*}
\begin{eqnarray*}
&\sum_{i=1}^{n-1}(-1)^im\otimes\displaystyle
\left(
\begin{array}{cccccccc}
 a_{0}&b_{0,1}&...& b_{0,i-1}b_{0,i} &  ...&b_{0,n-1}\\
1 & a_{1}&...&b_{1,i-1}b_{1,i}    &...&b_{1,n-1}\\
. &.& ...       &.&...&.\\
1& 1&...&a_{i-1}a_{i}\varepsilon(b_{i-1,i})& ...&b_{i-1,n-1}b_{i,n-1}\\
. &.&...& .       &...&.\\
1 &1& ...&1& ...&a_{n-1}\\
\end{array}
\right)+&\\
&(-1)^{n}a_{n-1}m\varepsilon(b_{0,n-1}b_{1,n-1}...b_{n-2,n-1})\displaystyle\otimes
\left(
\begin{array}{cccccccc} a_{0}&b_{0,1}&...& b_{0,i} &  ...&b_{0,n-1}\\
1 & a_1&...&b_{1,i}    &...&b_{1,n-1}\\
. &.&...& .       &...&.\\
1& .&...&a_i& ...&b_{i,n-1}\\
. &.&...& .       &...&.\\
1 & .&...&1& ...&a_{n-1}\\

\end{array}
\right).&
\end{eqnarray*}

In particular, we get that $H_0((A,B, \varepsilon);M)=H_0(A,M)$.
\end{remark}

\begin{example} We introduce the simplicial $\mathcal{A}(A,B,\varepsilon)$ right module $\mathcal{L}(A,B,\varepsilon)$ by setting $L_n=A\otimes B^{\otimes n}$. The $A_n$-module structure on $L_n$ is given by
\begin{eqnarray*}
(a_0\otimes b_1\otimes ...\otimes b_n)\cdot (a\otimes \alpha_1\otimes ...\otimes \alpha_{n}\otimes\gamma\otimes \beta_1\otimes...\otimes \beta_n\otimes b)=\\
ba_0a\varepsilon(\gamma)\otimes \alpha_1b_1\beta_1\otimes ...\otimes \alpha_nb_n\beta_n
\end{eqnarray*}
The simplicial structure maps, $\delta_i^{\mathcal{L}}:L_n\to L_{n-1}$ and $\sigma_i^{\mathcal{L}}:L_n\to L_{n+1}$, are defined  by
\begin{eqnarray*}
\delta_0^{\mathcal{L}}(a\otimes b_1\otimes ...\otimes b_{n})&=&a\varepsilon(b_1)\otimes b_2\otimes ...\otimes b_{n},\\
\delta_{n}^{\mathcal{L}}(a\otimes b_1\otimes ...\otimes b_{n})&=&\varepsilon(b_{n})a\otimes b_1\otimes ...\otimes b_{n-1},\\
\delta_i^{\mathcal{L}}(a\otimes b_1\otimes ...\otimes b_{n})&=&a\otimes b_1\otimes ...\otimes b_ib_{i+1}\otimes...\otimes b_{n},
\end{eqnarray*}
if $1\leq i\leq n-1$, and \begin{eqnarray*}
\sigma_i^{\mathcal{L}}(a\otimes b_1\otimes ...\otimes b_{n})=a\otimes b_1\otimes...\otimes b_{i}\otimes 1\otimes b_{i+1}\otimes  ...\otimes b_{n},
\end{eqnarray*}
if $0\leq i\leq n$. \label{exh2}
\end{example}
\begin{proof}
Straightforward computation.
\end{proof}

\begin{definition}
The homology of the complex associated to the simplicial $k$-module $\mathcal{L}(A,B,\varepsilon)\otimes_{\mathcal{A}(A,B,\varepsilon)} \mathcal{B}(A,B,\varepsilon)$ is called the secondary Hochschild homology associated to the triple $(A,B,\varepsilon)$  and it is denoted by  $HH_{\bullet}(A,B,\varepsilon)$.
\end{definition}

\begin{remark} We can make the identification  $L_n\otimes_{A_n}B_n=A^{\otimes (n+1)}\otimes B^{\otimes \frac{n(n+1)}{2}}$. If we set $C_n(A, B, \varepsilon):=A^{\otimes (n+1)}\otimes B^{\otimes \frac{n(n+1)}{2}}$ and denote the induced differential by $b$, one can see that $HH_\bullet(A, B, \varepsilon)$ is the homology of the complex $(C_\bullet(A, B, \varepsilon), b)$, where \begin{eqnarray*}
&b_n\left(\displaystyle\otimes
\left(
\begin{array}{cccccccc}
 a_{0}& b_{0,1} &  ...&b_{0,n-1}&b_{0,n}\\
1 & a_{1}    &...&b_{1,n-1}&b_{1,n}\\
. & .       &...&.&.\\
1& 1& ...&a_{n-1}&b_{n-1,n}\\
1 & 1& ...&1&a_{n}\\
\end{array}
\right)\right)=&\\ \end{eqnarray*}
\begin{eqnarray*}
&\sum_{i=0}^{n-1}(-1)^i\left(\displaystyle\otimes
\left(
\begin{array}{cccccccc}
 a_{0}&...& b_{0,i}b_{0,i+1} &  ...&b_{0,n-1}&b_{0,n}\\
1 & ...&b_{1,i}b_{1,i+1}    &...&b_{1,n-1}&b_{1,n}\\
. &...& .       &...&.&.\\
1& ...&a_ia_{i+1}\varepsilon(b_{i,i+1})& ...&b_{i,n-1}b_{i+1,n-1}&b_{i,n}b_{i+1,n}\\
. &...& .       &...&.&.\\
1 & ...&1& ...&1&a_{n}\\
\end{array}
\right)\right)+&\\
&(-1)^{n}\left(\displaystyle\otimes
\left(
\begin{array}{cccccccc} a_{n}a_{0}\varepsilon(b_{0,n})&b_{1,n}b_{0,1}&...& b_{i,n}b_{0,i} &  ...&b_{n-1,n}b_{0,n-1}\\
1 & a_1&...&b_{1,i}    &...&b_{1,n-1}\\
. &.&...& .       &...&.\\
1& .&...&a_i& ...&b_{i,n-1}\\
. &.&...& .       &...&.\\
1 & .&...&1& ...&a_{n-1}\\

\end{array}
\right)\right).&
\end{eqnarray*}\end{remark}

\begin{remark}
If $A$ is a commutative algebra, $HH_0(A,B,\varepsilon)=HH_0(A)$.
\end{remark}

Just like in the case of $C^n(A, B, \varepsilon)$, there are natural left and right actions of the cyclic group $C_{n+1}=<\lambda>$ on $C_n(A, B, \varepsilon)$. More precisely, the left action is given by
\begin{eqnarray*} &\lambda \cdot\left(\displaystyle\otimes
\left(
\begin{array}{cccccccc}
 a_{0}& b_{0,1} &b_{0,2}&  ...&b_{0,n-1}&b_{0,n}\\
1 & a_{1}    &b_{1,2}&...&b_{1,n-1}&b_{1,n}\\
. & .       &.&...&.&.\\
1& 1& 1&...&a_{n-1}&b_{n-1,n}\\
1 & 1& 1&...&1&a_{n}\\
\end{array}
\right)\right)=\\
&(-1)^{n}\left(\displaystyle\otimes
\left(
\begin{array}{cccccccc}
 a_{n}& b_{0,n} &b_{1,n}&  ...&b_{n-2,n}&b_{n-1,n}\\
1 & a_{0}    &b_{0,1}&...&b_{0,n-2}&b_{0,n-1}\\
. & .       &.&...&.&.\\
1& 1& 1&...&a_{n-2}&b_{n-2,n-1}\\
1 & 1& 1&...&1&a_{n-1}\\
\end{array}
\right)\right).&
\end{eqnarray*}

We define the complex $C_n^\lambda(A, B, \varepsilon):=C_n(A, B, \varepsilon)/Im(1-\lambda)$. If $b'_n$ denotes the sum of the first $n$ terms in the formula defining $b_n$, the relation $(1-\lambda)b'_n=b_n(1-\lambda)$ holds, so the operator $b$ is well defined on $C_\bullet^\lambda(A, B, \varepsilon)$. With these considerations, we introduce the cyclic homology associated to the triple $(A, B, \varepsilon)$.

\begin{definition} The  homology of the complex $(C_\bullet^\lambda(A, B, \varepsilon), b)$ is denoted by $HC_{\bullet}(A,B,\varepsilon)$ and is called the
cyclic homology associated to the triple $(A, B, \varepsilon)$.
\end{definition}

\begin{remark} If $n=0$, $HC_0(A, B, \varepsilon)=A/[A, A]$.\end{remark}

Similar to the case of the classical cyclic homology, we have a long exact sequence which relates the secondary Hochschild and cyclic homology groups associated to the triple $(A, B, \varepsilon)$. Our approach in proving this result is based on a natural extension of the cyclic bicomplex of $A$. Then we follow the same line of proof as in \cite{k2}, so we won't reproduce all the details here.

For $N=1+\lambda+\lambda^2+\cdots+\lambda^n$, we have the relations $N(1-\lambda)=(1-\lambda)N=0$, $Nb=b'N,$ $(1-\lambda)b'=b(1-\lambda)$ and $b^2=b'^2=0$. This implies that we have the following double complex, which we denote by $C(A, B, \varepsilon)$.
$$\xymatrix{\ar[d]_b&\ar[d]_{-b'}&\ar[d]_b\\
C_2(A, B, \varepsilon)\ar[d]_b&C_2(A, B, \varepsilon)\ar[l]_-{1-\lambda}\ar[d]_{-b'}&C_2(A, B, \varepsilon)\ar[l]_{N}\ar[d]_{b'}&\cdots\ar[l]_-{1-\lambda}\\
C_1(A, B, \varepsilon)\ar[d]_b&C_1(A, B, \varepsilon)\ar[d]_{-b'}\ar[l]_-{1-\lambda}&C_1(A, B, \varepsilon)\ar[d]_{b}\ar[l]_-N&\cdots\ar[l]_-{1-\lambda}\\
C_0(A, B, \varepsilon)&C_0(A, B, \varepsilon)\ar[l]_{1-\lambda}&C_0(A, B, \varepsilon)\ar[l]\ar[l]_N&\cdots\ar[l]_-{1-\lambda}}$$

In characteristic 0, we have that $Ker(1-\lambda)= Im N$ and $Ker N=Im(1-\lambda)$ so the rows are exact. It follows that the homology of the total complex, $(TotC(A, B, \varepsilon), \delta)$, is exactly the cyclic homology associated to the triple $(A, B, \varepsilon)$. Moreover, following \cite{k2}, we have a short exact sequence of complexes $$0\to Tot'_\bullet C(A,B,\varepsilon)\stackrel{i}{\rightarrow} Tot_\bullet C(A,B,\varepsilon)[2]\stackrel{s}{\rightarrow} Tot_\bullet C(A,B,\varepsilon)\to 0,$$
where $Tot_\bullet C(A, B, \varepsilon)[2]$ is the shifted by 2 complex and the truncation map $s$ is defined by $$s(x_0,x_1,\cdots, x_n)=(x_0, x_1,\cdots x_{n-2}).$$ The kernel of $s$ is the total complex of the first two columns of $C(A, B, \varepsilon),$ $$Tot'_nC(A, B, \varepsilon)=A^{\otimes n}\otimes B^{\otimes\frac{(n-1)n}{2}}\oplus A^{\otimes(n-1)}\otimes B^{\otimes\frac{(n-2)(n-1)}{2}} .$$ It is not hard to see that the complexes $(C_\bullet(A, B, \varepsilon), b)$ and $Tot_\bullet' C(A, B, \varepsilon)$ have the same homology since they are homotopy equivalent via the map \begin{center} $t_n\left(\displaystyle\otimes
\left(
\begin{array}{cccccccc}
 a_{0}& b_{0,1} &  ...&b_{0,n-1}\\
1 & a_{1}       & ...&b_{1,n-1}\\
. & .       &...&.\\
1 & 1& ...&a_{n-1}\\
\end{array}
\right)\right)=\displaystyle\otimes
\left(
\begin{array}{cccccccc}
1  & 1 & 1 &...&1\\
1& a_{0}& b_{0,1} &  ...&b_{0,n-1}\\
1& 1 & a_{1}       & ...&b_{1,n-1}\\
. & . &.      &...&.\\
1 & 1& .&...&a_{n-1}\\

\end{array}
\right). $\end{center}

We obtain that the short exact sequence defined above induces the desired long exact sequence. More precisely, we have the following theorem.

\begin{theorem}\label{Connes2} Let $k$ be a field of characteristic zero. For a triple $(A,B,\varepsilon)$ we have a long exact sequence
$$...\rightarrow HC_n(A,B,\varepsilon)\stackrel{S}{\rightarrow}  HC_{n-2}(A,B,\varepsilon)\stackrel{B}{\rightarrow}  HH_{n-1}(A,B,\varepsilon)\stackrel{I}{\rightarrow}  HC_{n-1}(A,B,\varepsilon)\rightarrow ....$$
\end{theorem}

Here $S$ is induced by the truncation map $s$ and $B$ is the connecting homomorphism of the long exact sequence. An explicit formula for $B$, at the level of chains, is $B=Nt(1-\lambda):C_n(A, B, \varepsilon)\rightarrow C_{n+1}(A, B, \varepsilon)$.

\begin{remark}
In a future paper we are planing to come back to these groups and discus their relation with the de Rham homology. We are especially interested in defining a $K$-theory associated to a triple $(A,B,\varepsilon)$, and in finding its geometrical meaning. Notice that even in the commutative case the construction from this paper has the potential to provide a new perspective on problems in geometry.
\end{remark}




\bibliographystyle{amsalpha}

\end{document}